\documentclass{article}
\usepackage{amsmath,amssymb,amsthm}
\usepackage[cp1250]{inputenc}
\usepackage{enumerate}
\usepackage{a4wide}
\usepackage{xcolor}

\numberwithin{equation}{section}

\usepackage{graphicx}
\usepackage{subfig}

\newcommand{\ep}{\varepsilon}
\newcommand{\ept}{\tilde{\varepsilon}}
\newcommand{\Et}{\tilde{E}}
\newcommand{\id}{\mathbb{I}}
\newcommand{\Sz}{\mathcal{S}^{2\times 2}_0}
\newcommand{\R}{\mathbb{R}} 
\newcommand{\U}{\mathbb{U}}
\newcommand{\Div}{{\rm div}_\vx} 
\newcommand{\vr}{\varrho}
\newcommand{\half}{\frac{1}{2}} 
\newcommand{\vc}[1]{{\bf #1}}
\newcommand{\vv}{\vc{v}}
\newcommand{\vx}{\vc{x}} 
\newcommand{\va}{\vc{a}} 
\newcommand{\Grad}{\nabla_\vx}
\newcommand{\dx}{\,{\rm d}\vx}
\newcommand{\dt}{\,{\rm d}t}
\newcommand{\intx}[1]{\int_{\R^2}#1\dx}
\newcommand{\intt}[1]{\int_0^\infty #1\dt}
\newcommand{\DC}{C^\infty_{\rm c}} 
\newcommand{\bfphi}{\boldsymbol{\varphi}}
\newcommand{\ov}[1]{\overline{#1}} 
\newcommand{\un}[1]{\underline{#1}} 
\newcommand{\til}[1]{\widetilde{#1}} 
\newcommand{\new}{{\rm new}}
\newcommand{\trans}{\top}

\newtheorem{theorem}{Theorem}[section]
\newtheorem{lemma}[theorem]{Lemma}
\newtheorem{definition}[theorem]{Definition}
\newtheorem{propo}[theorem]{Proposition}
\newtheorem{remark}[theorem]{Remark}

\title{Shocks Make the Riemann Problem for the Full Euler System in Multiple Space Dimensions Ill-Posed}
\author{Christian Klingenberg$^2$\footnote{klingen@mathematik.uni-wuerzburg.de} \and Ond\v rej Kreml$^1$\footnote{kreml@math.cas.cz} \and
V\'aclav M\'acha$^1$\footnote{macha@math.cas.cz} \and Simon Markfelder$^2$\footnote{simon.markfelder@mathematik.uni-wuerzburg.de}}

\date{}

\begin{document}

\maketitle

\medskip

\centerline{$^1$Institute of Mathematics of the Czech Academy of Sciences}

\centerline{\v Zitn\' a 25, CZ-115 67 Praha 1, Czech Republic}
\medskip

\centerline{$^2$Department of Mathematics, W\"urzburg University}

\centerline{Emil-Fischer-Str. 40, 97074 W\"urzburg, Germany}

\bigskip

\textbf{Abstract:} The question of (non-)uniqueness of one-dimensional self-similar solutions to the Riemann problem for hyperbolic systems of gas dynamics in the class of multi-dimensional admissible weak solutions was addressed in recent years in several papers culminating in \cite{MarKli17} with the proof that the Riemann problem for the isentropic Euler system with a power law pressure is ill-posed if the one-dimensional self-similar solution contains a shock. Then the natural question arises whether the same holds also for a more involved system of equations, the full Euler system. After the first step in this direction was made in \cite{AKKMM}, where ill-posedness was proved in the case of two shocks appearing in the self-similar solution, we prove in this paper that the presence of just one shock in the self-similar solution implies the same outcome, i.e. the existence of infinitely many admissible weak solutions to the multi-dimensional problem.


\section{Introduction}

The full compressible Euler system in the whole two-dimensional space can be written as a system of partial differential equations 
\begin{equation}\left.
	\begin{split}
		\partial_t \vr + \Div (\vr\vv) &= 0 \\
		\partial_t (\vr\vv) + \Div(\vr\vv\otimes \vv)+\Grad p  &= 0 \\
		\partial_t \bigg(\half\vr|\vv|^2 + \vr e(\vr,p)\bigg) + \Div \bigg[\bigg(\half\vr|\vv|^2 + \vr e(\vr,p) + p\bigg)\vv\bigg] &= 0
	\end{split}
	\right\} \text{ in } (0,\infty) \times \R^2,
	\label{eq:euler}
\end{equation}
with the unknown density $\vr=\vr(t,\vx)\in\R^+$, pressure $p=p(t,\vx)\in\R^+$ and velocity field $\vv=\vv(t,\vx)\in\R^2$. The independent variables here are the time $t\in[0,\infty)$ and the position $\vx=(x,y)\in\R^2$. 

Throughout this paper we consider an ideal gas, i.e. the relation between the internal energy $e(\vr,p)$ and the density and the pressure is given by
\[
e(\vr,p) = c_v \frac{p}{\vr},
\]
where $c_v>0$ is a constant called the \emph{specific heat at constant volume}. The principle of equipartition of energy predicts $c_v=\frac{f}{2}$, where $f$ is the number of degrees of freedom of the gas particles.

We complement the system of equations \eqref{eq:euler} with the initial condition
\begin{equation}\label{eq:initial}
(\vr,\vv,p)(0,\vx) = (\vr^0,\vv^0,p^0)(\vx) \qquad \text{ in } \R^2.
\end{equation}
Moreover, we add to the Euler system \eqref{eq:euler} the entropy condition 
\begin{equation} \label{eq:entropy}
	\partial_t\Big(\vr s(\vr,p)\Big) + \Div\Big(\vr s(\vr,p)\vv\Big) \geq 0,
\end{equation}
where $s(\vr,p)$ denotes the (physical) entropy, which for the ideal gas takes the form
\begin{equation}\label{eq:entropyex}
s(\vr,p)=\log\Big(\frac{p^{c_v}}{\vr^{c_v+1}}\Big).
\end{equation}
Let us note here that in the terminology of hyperbolic conservation laws, the (mathematical) entropy is the quantity $-\vr s(\vr,p)$. We call weak solutions to the Euler system \eqref{eq:euler}-\eqref{eq:initial} \emph{admissible}, if they satisfy the entropy inequality \eqref{eq:entropy}. For a detailed definition of an \emph{admissible weak solution} see Section \ref{ss:weak}.

We remark that it is possible to formulate the system \eqref{eq:euler} in terms of other unknown variables, for example with temperature $\vartheta$ instead of pressure $p$. We choose to work with the set of unknowns as is formulated in \eqref{eq:euler} mainly because of its convenience with respect to the related 1D theory. It is however clear that the temperature can be always easily reconstructed from the pressure and the density using the relation $\vartheta = \frac{p}{\varrho}$, because throughout this paper we always assume that the density $\vr$ is strictly positive.

We study the Riemann problem for the Euler system \eqref{eq:euler}-\eqref{eq:entropy}, i.e. we assume that the initial data take the form
\begin{equation}
	\begin{split}
		(\vr^0,\vv^0,p^0)(\vx):=\left\{
		\begin{array}[c]{ll}
			(\vr_-,\vv_-,p_-) & \text{ if }y<0 \\
			(\vr_+,\vv_+,p_+) & \text{ if }y>0
		\end{array}
		\right. ,
	\end{split}
\label{eq:Riemann}
\end{equation}
where $\vr_\pm\in\R^+$, $\vv_\pm\in\R^2$ and $p_\pm\in\R^+$ are constants. We write $\vv=(u,v)^\trans$ for the components of the velocity $\vv$. 

The Riemann problem is an important building block of the 1D theory of systems of hyperbolic conservation laws and hence it is discussed in detail in classical monographs in the field, e.g. \cite{Dafermos16}. It is well-known, that in one space dimension this problem admits self-similar solutions (i.e. solutions which depend on a single variable $\xi = \frac{x}{t}$) consisting of constant states connected by rarefaction waves, admissible shocks or contact discontinuities. When the Riemann problem is considered in multiple space dimensions, the one-dimensional self-similar solution still solves the appropriate system of equations in multi-D. However, additional (and in particular genuinely multi-dimensional) admissible solutions may arise.

The question whether one-dimensional self-similar solutions are unique in the class of multi-dimensional admissible weak solutions of the appropriate system has been studied in several previous works. Chen and Chen \cite{Chen} proved that for the isentropic Euler system as well as for the full Euler system solutions consisting only of rarefaction waves are indeed unique in the class of admissible weak solutions to the multi-D problem. Later on similar results were proved in \cite{FeiKre} in the case of the isentropic Euler system and \cite{FeiKreVas} in the case of the full Euler system.

On the non-uniqueness side, the $L^\infty$ convex integration theory developed by De Lellis and Sz\'ekelyhidi \cite{DLSz1}, \cite{DLSz2} was first used in the context of piecewise constant initial data in the work of Sz\'ekelyhidi \cite{sz} with the following result: Vortex sheet initial data for the incompressible Euler equations allow for the existence of infinitely many bounded weak solutions. For the isentropic compressible Euler system, Chiodaroli, De Lellis and Kreml \cite{ChiDelKre15} applied the above mentioned theory and showed the existence of Riemann initial data for which there exist infinitely many bounded admissible weak solutions. Moreover, these initial data were generated by a compression wave, thus proving the existence of Lipschitz initial data allowing for existence of infinitely many bounded admissible weak solutions. Ill-posedness of the isentropic Euler system with smooth initial data was proved in \cite{CKMS}, which concluded the line of research which studied the question of smoothness of initial data allowing for existence of infinitely many admissible solutions, see also \cite{DLSz2}, \cite{Ch} and \cite{Fe}. 

The study of the Riemann problem for the isentropic Euler system with a power law pressure continued with a series of papers \cite{ChiKre14}, \cite{ChiKre17}, \cite{MarKli17} and \cite{BrChKr}, ultimately concluding that whenever the Riemann initial data for the isentropic Euler system are such that the one-dimensional self-similar solution contains a shock, the problem is ill-posed and admits infinitely many bounded admissible weak solutions in more than one space dimension. Moreover, in \cite{MaKl2} it was shown that one can construct infinitely many admissible weak solutions satisfying the energy equality instead of just the energy inequality.

Interestingly, the question whether a self-similar solution to the isentropic Euler system with power law pressure consisting of a single contact discontinuity, i.e. jump in the first component of the velocity, is unique in the set of admissible multi-dimensional weak solutions remains to be open. This problem is a direct generalization of the above mentioned result of Sz\'ekelyhidi \cite{sz}. Non-uniqueness of solutions consisting of contact discontinuities was addressed in \cite{BrKrMa} for the isentropic Euler equations with the Chaplygin gas pressure law, which is a linearly degenerate hyperbolic system of conservation laws allowing only for contact discontinuities to appear in self-similar solutions. 

As already mentioned, in \cite{MaKl2} infinitely many admissible and energy conserving weak solutions for the isentropic Euler equations were constructed, which therefore solve the full Euler equations as well. Hence it is natural to turn the attention to the question of non-uniqueness of solutions for the full Euler system. This question was first studied in \cite{FeKlKrMa}, where it was shown that for piecewise constant initial density and temperature there exists a bounded initial velocity field allowing for existence of infinitely many solutions. The non-uniqueness of self-similar solutions to the Riemann problem for the full Euler system in multiple space dimensions was first studied in \cite{AKKMM}, where the authors proved that if the Riemann initial data produce a self-similar solution containing two shocks, there exist infinitely many additional admissible weak solutions to the multi-D problem.

In this paper we continue the research started in \cite{AKKMM} and prove the analogue of the result in \cite{MarKli17}, namely that the Riemann problem for the full Euler system is ill-posed whenever the self-similar solution contains a shock. Our main theorem reads as follows.

\begin{theorem}\label{t:main}
Let $c_v > \frac 12$. Let $(\vr_-,\vv_-,p_-)$ and $(\vr_+,\vv_+,p_+)$ be Riemann initial data as in \eqref{eq:Riemann}, such that the 1D self-similar solution to the Riemann problem \eqref{eq:euler}-\eqref{eq:Riemann} contains a shock. Then there exist infinitely many bounded admissible weak solutions to the Riemann problem for the problem \eqref{eq:euler}-\eqref{eq:Riemann}. 
\end{theorem}

\begin{remark}
	Throughout this paper we will always assume that $c_v>\frac 12$. This is justified since the number of degrees of freedom in a two-dimensional gas should be larger or equal 2. In particular we have $f>1$ and hence $c_v = \frac{f}{2}>\frac 12$. We will make use of this assumption at the end of the proof of Theorem \ref{t:small} in Section \ref{subsec:subsol_cond_2}. We were not able to get rid of this assumption.
\end{remark}

As in previous results concerning non-uniqueness of weak solutions to the Riemann problem, the proof of Theorem \ref{t:main} is based on the $L^\infty$ convex integration theory of De Lellis and Sz\'ekelyhidi \cite{DLSz1}, \cite{DLSz2} formulated for our purposes in Lemma \ref{l:key}. Fan subsolutions (see Definition \ref{d:ss}) are then designed in such a way that existence of a single fan subsolution implies directly the existence of infinitely many admissible weak solutions. The main task is then to prove existence of such a fan subsolution. Somewhat surprisingly it turns out that a general ansatz leading to the proof of Theorem \ref{t:main} under a certain smallness assumption, consists of setting up a fan subsolution with three interfaces, where the interface mimicking the contact discontinuity is not the middle one, as is the case in the self-similar solution, but the right one. For more details see Sections \ref{ss:ansatz} and \ref{s:conclude}. The final argument of the proof of Theorem \ref{t:main} is the patching procedure introduced in \cite{MarKli17}.

The paper is structured as follows. In Section \ref{s:Prel} we introduce all necessary preliminary material including the structure of 1D self-similar solutions, definitions of admissible weak solutions and admissible fan subsolutions, we provide Proposition \ref{p:ss} relating the existence of a single subsolution to the existence of infinitely many admissible weak solutions and present principles of invariance of Euler equations which allow us to work in a somewhat simplified setting. In Section \ref{s:small} we prove a smallness result which is the key building block in the proof of Theorem \ref{t:main}. The complete proof of Theorem \ref{t:main} is then the content of Section \ref{s:proof_main}. Finally, in Section \ref{s:conclude} we present some concluding remarks.

\section{Preliminaries}\label{s:Prel}

\subsection{Admissible Weak Solutions}\label{ss:weak} 

For completeness of presentation we provide here the definition of weak solutions we work with.

\begin{definition}
	The trio $(\vr,\vv,p)$ is called a bounded weak solution to the full compressible Euler system \eqref{eq:euler}-\eqref{eq:initial} if $(\vr,\vv,p) \in L^\infty([0,\infty) \times \R^2; \R^+ \times \R^2 \times \R^+)$, $\vr,p \geq 0$ and the following integral equations are satisfied. 
	\begin{equation*}
	\int_0^\infty \int_{\R^2} \left(\vr \partial_t \varphi + \vr\vv\cdot\nabla_\vx \varphi\right)\dx\dt + \int_{\R^2} \vr^0\varphi(0,\cdot) \dx = 0
	\end{equation*}
	for all $\varphi \in C^\infty_c([0,\infty) \times \R^2)$,
	\begin{equation*}
	\int_0^\infty \int_{\R^2} \left(\vr\vv \cdot\partial_t \bfphi + \vr\vv\otimes\vv:\nabla_\vx \bfphi + p\Div \bfphi\right)\dx\dt + \int_{\R^2} \vr^0\vv^0\cdot\bfphi(0,\cdot) \dx = 0
	\end{equation*}
	for all $\bfphi \in C^\infty_c([0,\infty) \times \R^2;\R^2)$,
	\begin{align*}
	&\int_0^\infty \int_{\R^2} \left(\bigg(\half\vr|\vv|^2 + \vr e(\vr,p)\bigg) \partial_t \psi + \bigg(\half\vr|\vv|^2 + \vr e(\vr,p) + p\bigg)\vv\cdot\nabla_\vx \psi \right)\dx\dt \\ 
	& \qquad + \int_{\R^2} \bigg(\half\vr^0|\vv^0|^2 + \vr^0 e(\vr^0,p^0)\bigg)\psi(0,\cdot) \dx = 0 
	\end{align*}
	for all $\psi \in C^\infty_c([0,\infty) \times \R^2)$.
	
	The bounded weak solution is called admissible, if moreover
	\begin{equation*}
	\int_0^\infty \int_{\R^2} \left(\vr s(\vr,p) \partial_t \phi + \vr s(\vr,p) \vv\cdot\nabla_\vx \phi\right)\dx\dt + \int_{\R^2} \vr^0 s(\vr^0,p^0)\phi(0,\cdot) \dx \leq 0
	\end{equation*}
	for all $\phi \in C^\infty_c([0,\infty) \times \R^2)$, $\phi \geq 0$.
\end{definition}

\subsection{1D Self-Similar Solutions}\label{ss:1D}

In order to study the structure of one-dimensional self-similar solutions to the 1D Riemann problem for the Euler equations we consider the 1D version of \eqref{eq:euler} which is
\begin{equation}
	\begin{split}
		\partial_t \vr + \partial_y (\vr v) &= 0, \\
		\partial_t (\vr v) + \partial_y \left(\vr v^2 + p \right)  &= 0, \\
		\partial_t \bigg(\half\vr v^2 + \vr e(\vr,p)\bigg) + \partial_y \bigg[\bigg(\half\vr v^2 + \vr e(\vr,p) + p\bigg)v\bigg] &= 0.
	\end{split}
	\label{eq:euler1D}
\end{equation}
The unknowns in \eqref{eq:euler1D} are the density $\vr=\vr(t,y)\in\R^+$, pressure $p=p(t,y)\in\R^+$ and (scalar) velocity $v=v(t,y)\in\R$, where the independent variables are again the time $t\in[0,\infty)$ and the (scalar) position $y\in\R$.

The entropy condition is now formulated as
\begin{equation} \label{eq:entropy1D}
	\partial_t\Big(\vr s(\vr,p)\Big) + \partial_y \Big(\vr s(\vr,p) v\Big) \geq 0
\end{equation}
and we study the Riemann problem with initial data
\begin{equation}
	\begin{split}
		(\vr^0,v^0,p^0)(y):=\left\{
		\begin{array}[c]{ll}
			(\vr_-,v_-,p_-) & \text{ if }y<0 \\
			(\vr_+,v_+,p_+) & \text{ if }y>0
		\end{array}
		\right. ,
	\end{split}
\label{eq:Riemann1D}
\end{equation}
where $\vr_\pm\in\R^+$, $v_\pm\in\R$ and $p_\pm\in\R^+$ are constants. 

It is well-known, that the system \eqref{eq:euler1D} possesses three characteristic families corresponding to the following eigenvalues
\begin{equation} \label{eq:eigenvalues}
    \lambda_1 = v - \sqrt{\frac{c_v+1}{c_v}\frac{p}{\vr}}, \qquad \lambda_2 = v, \qquad \lambda_3 = v + \sqrt{\frac{c_v+1}{c_v}\frac{p}{\vr}}.
\end{equation}
The $1$- and $3$-families are genuinely non-linear and therefore produce either rarefaction waves or admissible shocks. On the other hand the $2$-family is linearly degenerate and produces contact discontinuities. It is well-known that the velocity $v$ and the pressure $p$ are constant across the contact discontinuity, whereas the discontinuity only appears in the density $\vr$.
There are 18 possible structures of 1D self-similar solutions to \eqref{eq:euler1D}, which we present in Table \ref{table}.
\renewcommand{\arraystretch}{1.3}
\begin{table}[h]
	\centering 
	\begin{tabular}{|c|c|c|c|c|c|c|c|c|} \cline{2-4} \cline{7-9}
		\multicolumn{1}{c|}{}& \centering 1-wave & \centering 2-wave & \centering 3-wave & \multicolumn{2}{c|}{} & \centering 1-wave & \centering 2-wave & \centering 3-wave \tabularnewline \cline{2-4} \cline{7-9} \multicolumn{7}{c}{}\\[-4.7mm] \cline{1-4} \cline{6-9}
		1 & \centering - & \centering - & \centering - & & 10 & \centering - & \centering contact & \centering - \tabularnewline \cline{1-4} \cline{6-9} 
		2 & \centering - & \centering - & \centering shock & & 11 & \centering - & \centering contact & \centering shock \tabularnewline \cline{1-4} \cline{6-9}
		3 & \centering - & \centering - & \centering rarefaction & & 12 & \centering - & \centering contact & \centering rarefaction \tabularnewline \cline{1-4} \cline{6-9}
		4 & \centering shock & \centering - & \centering - & & 13 & \centering shock & \centering contact & \centering - \tabularnewline \cline{1-4} \cline{6-9} 
		5 & \centering shock & \centering - & \centering shock & & 14 & \centering shock & \centering contact & \centering shock \tabularnewline \cline{1-4} \cline{6-9}
		6 & \centering shock & \centering - & \centering rarefaction & & 15 & \centering shock & \centering contact & \centering rarefaction \tabularnewline \cline{1-4} \cline{6-9}
		7 & \centering rarefaction & \centering - & \centering - & & 16 & \centering rarefaction & \centering contact & \centering - \tabularnewline \cline{1-4} \cline{6-9}
		8 & \centering rarefaction & \centering - & \centering shock & & 17 & \centering rarefaction & \centering contact & \centering shock \tabularnewline \cline{1-4} \cline{6-9}
		9 &\centering rarefaction & \centering - & \centering rarefaction & & 18 & \centering rarefaction & \centering contact & \centering rarefaction \tabularnewline \cline{1-4} \cline{6-9}
	\end{tabular}
	\caption{All the 18 possibilities of the structure of the 1D Riemann solution} \label{table}
\end{table}
\renewcommand{\arraystretch}{1}

We are interested in this paper in cases containing exactly one shock, which are cases 2, 4, 6, 8, 11, 13, 15 and 17 in Table \ref{table}.

In Proposition \ref{p:1D} we present the conditions on the Riemann initial data such that the 1D self-similar solution contains exactly one shock.

\begin{propo}\label{p:1D}
\begin{itemize}
    \item     Assume that  $p_- < p_+$ and
    \begin{equation}\label{eq:SCDRcon1} 
        -2\sqrt{c_v(c_v+1)}\sqrt{\frac{p_+}{\vr_+}}\left(1-\left(\frac{p_-}{p_+}\right)^\frac{1}{2(c_v+1)} \right) < v_- - v_+ < (p_+ - p_-)\sqrt{\frac{2c_v}{\vr_-(p_-+(2c_v+1)p_+)}}.
    \end{equation}
    Then the 1D Riemann solution to the problem \eqref{eq:euler1D}-\eqref{eq:Riemann1D} consists of a 1-shock, a possible 2-contact discontinuity and a 3-rarefaction wave.
    
    \item     Assume that  $p_- < p_+$ and
    \begin{equation}\label{eq:SCDRcon1a}
        v_- - v_+ = (p_+ - p_-)\sqrt{\frac{2c_v}{\vr_-(p_-+(2c_v+1)p_+)}}.
    \end{equation}
    Then the 1D Riemann solution to the problem \eqref{eq:euler1D}-\eqref{eq:Riemann1D} consists of a 1-shock and a possible 2-contact discontinuity.
    
    \item  Assume that  $p_- > p_+$ and
    \begin{equation}\label{eq:SCDRcon2}
        -2\sqrt{c_v(c_v+1)}\sqrt{\frac{p_-}{\vr_-}}\left(1-\left(\frac{p_+}{p_-}\right)^\frac{1}{2(c_v+1)}\right) < v_- - v_+ < (p_- - p_+)\sqrt{\frac{2c_v}{\vr_+(p_++(2c_v+1)p_-)}}.
    \end{equation}
    Then the 1D Riemann solution to the problem \eqref{eq:euler1D}-\eqref{eq:Riemann1D} consists of a 1-rarefaction wave, a possible 2-contact discontinuity and a 3-shock.
    
    \item  Assume that  $p_- > p_+$ and
    \begin{equation}\label{eq:SCDRcon2a}
        v_- - v_+ = (p_- - p_+)\sqrt{\frac{2c_v}{\vr_+(p_++(2c_v+1)p_-)}}.
    \end{equation}
    Then the 1D Riemann solution to the problem \eqref{eq:euler1D}-\eqref{eq:Riemann1D} consists of a possible 2-contact discontinuity and a 3-shock.
    \end{itemize}
\end{propo}
For the proof of Proposition \ref{p:1D} see \cite[Theorem 18.7]{Smoller67}.

It is not difficult to observe that the situation does not change significantly if we consider one-dimensional solutions to the two-dimensional problem \eqref{eq:euler}. Admissible solutions then satisfy
\begin{equation}
	\begin{split}
		\partial_t \vr + \partial_y (\vr v) &= 0, \\
		\partial_t (\vr u) + \partial_y (\vr uv) &=0, \\
		\partial_t (\vr v) + \partial_y (\vr v^2 + p )  &= 0, \\
		\partial_t \bigg(\half\vr(u^2 + v^2) + \vr e(\vr,p)\bigg) + \partial_y \bigg[\bigg(\half\vr (u^2 + v^2) + \vr e(\vr,p) + p\bigg)v\bigg] &= 0.
	\end{split}
	\label{eq:euler1D2D}
\end{equation} 
with
\begin{equation} \label{eq:entropy1D2D}
	\partial_t\Big(\vr s(\vr,p)\Big) + \partial_y \Big(\vr s(\vr,p) v\Big) \geq 0
\end{equation}
and we are interested in Riemann initial data
\begin{equation}
	\begin{split}
		(\vr^0,u^0,v^0,p^0)(y):=\left\{
		\begin{array}[c]{ll}
			(\vr_-,u_-,v_-,p_-) & \text{ if }y<0 \\
			(\vr_+,u_+,v_+,p_+) & \text{ if }y>0
		\end{array}
		\right. .
	\end{split}
\label{eq:Riemann1D2D}
\end{equation}

System \eqref{eq:euler1D2D} possesses also three distinct eigenvalues as in \eqref{eq:eigenvalues}, where the second eigenvalue $\lambda_2$ has however now multiplicity $2$. It is not difficult to observe, that the structure of self-similar solutions to the Riemann problem for system \eqref{eq:euler1D2D} is the same as the structure of self-similar solutions to \eqref{eq:euler1D}. The only difference is that in the case $u_- \neq u_+$ the first component of the velocity of the solution exhibits a jump from the value $u_-$ to $u_+$ on the same $2$-contact discontinuity where the density $\vr$ jumps.

In particular Proposition \ref{p:1D} holds also for the problem \eqref{eq:euler1D2D}-\eqref{eq:Riemann1D2D} regardless of values $u_{\pm}$.

\subsection{Subsolutions}\label{ss:subs} 

The main tool in our proof of Theorem \ref{t:main} is the notion of a fan subsolution, which was introduced in \cite{ChiDelKre15} and extended to the full Euler system in \cite{AKKMM}. For the sake of completeness we reformulate the definition. Furthermore we recap the relation between the existence of one fan subsolution and the existence of infinitely many admissible weak solutions to \eqref{eq:euler}-\eqref{eq:Riemann}.
\begin{definition}
Let $\mu_0 < \mu_1 < \mu_2$ be real numbers. A fan partition of $(0,\infty)\times \R^2$ consists of four sets $\Omega_-,\Omega_1,\Omega_2, \Omega_+$ of the form
\begin{align*} 
    \Omega_- &= \left\{(t,\vx) : t > 0 \text{ and } y < \mu_0 t\right\} \\
    \Omega_1 \:&= \left\{(t,\vx) : t > 0 \text{ and } \mu_0 t < y < \mu_1 t\right\} \\
    \Omega_2 \:&= \left\{(t,\vx) : t > 0 \text{ and } \mu_1 t < y < \mu_2 t\right\} \\
    \Omega_+ &= \left\{(t,\vx) : t > 0 \text{ and } \mu_2 t < y \right\}.
\end{align*}
\end{definition}

We remark that in \cite{AKKMM} this object was called a $2$-fan partition. In what follows, the symbol $\Sz$ denotes the set of symmetric traceless $2\times2$ matrices. Now we are ready to define a fan subsolution.

\begin{definition} 
	An admissible fan subsolution to the Riemann problem for the Euler system \eqref{eq:euler}-\eqref{eq:Riemann} is a quintuple $(\ov{\vr},\ov{\vv},\ov{\U},\ov{C},\ov{p}):(0,\infty)\times\R^2\rightarrow(\R^+\times\R^2\times\Sz\times\R^+\times\R^+)$ of piecewise constant functions, which satisfies the following properties:
	\begin{enumerate}
		\item There exists a fan partition  $\Omega_-,\Omega_1,\Omega_2,\Omega_+$ of $(0,\infty)\times\R^2$ and for $i\in\{1,2\}$ there exist constants $\vr_i\in\R^+$, $\vv_i\in\R^2$, $\U_i\in\Sz$, $C_i\in\R^+$ and $p_i\in\R^+$, such that
		\begin{align*}
			(\ov{\vr},\ov{\vv},\ov{\U},\ov{C},\ov{p})&=\sum\limits_{i\in\{-,+\}} \bigg(\vr_i\,,\,\vv_i\,,\,\vv_i\otimes \vv_i - \frac{|\vv_i|^2}{2}\,\id\,,\,|\vv_i|^2\,,\,p_i\bigg)\,\mathbf{1}_{\Omega_i} + \sum\limits_{i=1}^2 (\vr_i\,,\,\vv_i\,,\,\U_i\,,\,C_i\,,\,p_i)\,\mathbf{1}_{\Omega_i},
		\end{align*}
		where $\vr_{\pm},\vv_{\pm},p_\pm$ are the constants given by the initial condition \eqref{eq:Riemann}.
		
		\item For $i\in\{1,2\}$ the following inequality holds in the sense of definiteness:
		\begin{equation*}
			\vv_i\otimes \vv_i-\U_i < \frac{C_i}{2} \id.
		\end{equation*}
		\item For all test functions $(\varphi,\bfphi,\psi)\in \DC([0,\infty)\times\R^2;\R\times\R^2\times\R)$ the following identities hold:
		\begin{align*} 
			\intt{\intx{\big[\ov{\vr} \partial_t\varphi + \ov{\vr} \ov{\vv}\cdot\Grad\varphi\big]}} + \intx{ \vr^0 \varphi(0,\cdot)}\ &=\ 0, \\
			\int_0^\infty\int_{\R^2}\bigg[\ov{\vr} \ov{\vv}\cdot\partial_t\bfphi + \ov{\vr} \ov{\U}:\Grad\bfphi + \bigg(\ov{p} + \half \ov{\vr} \ov{C} \bigg) \Div\bfphi\bigg]\dx\dt  + \intx{ \vr^0 \vv^0\cdot\bfphi(0,\cdot)}\  &=\ 0, \\
			\int_0^\infty\int_{\R^2}\bigg[\bigg(\half \ov{\vr} \ov{C} + c_v \ov{p}\bigg) \partial_t\psi +\bigg(\half \ov{\vr} \ov{C} + (c_v+1) \ov{p}\bigg) \ov{\vv}\cdot\Grad\psi\bigg]\dx\dt & \\
			+ \intx{\bigg(\vr^0 \frac{|\vv^0|^2}{2} + c_v p^0\bigg) \psi(0,\cdot)} \ &=\ 0. 
		\end{align*}
		\item For every non-negative test function $\phi\in \DC([0,\infty)\times\R^2;\R_0^+)$ the inequality
		\begin{align*}
			&\intt{\intx{\Big[\ov{\vr} s(\ov{\vr},\ov{p}) \partial_t\phi + \ov{\vr} s(\ov{\vr},\ov{p}) \ov{\vv}\cdot\Grad\phi\Big]}} + \intx{ \vr^0 s(\vr^0,p^0) \phi(0,\cdot)} \leq 0 
		\end{align*}
		is fulfilled.
	\end{enumerate}
	\label{d:ss}
\end{definition}

The relation between the notion of subsolution and existence of infinitely many solutions to the Euler equations is stated in the following Proposition.

\begin{propo}
	Let $(\vr_\pm,\vv_\pm,p_\pm)$ be such that there exists an admissible fan subsolution $(\ov{\vr},\ov{\vv},\ov{\U},\ov{C},\ov{p})$ to the Cauchy problem \eqref{eq:euler}-\eqref{eq:Riemann}. Then there exist infinitely many bounded admissible weak solutions $(\vr,\vv,p)$ to \eqref{eq:euler}-\eqref{eq:Riemann} with the following properties:
	\begin{itemize}
		\item $(\vr,p)=(\ov{\vr},\ov{p})$,
		\item $\vv= \vv_\pm$ on $\Omega_\pm$, 
		\item $|\vv|^2= C_i$ a.e. on $\Omega_i$, $i=1,2$.
	\end{itemize}
	\label{p:ss}
\end{propo}

The proof of Proposition \ref{p:ss} is based on the $L^\infty$ theory developed by De Lellis and Sz\'ekelyhidi \cite{DLSz1}, \cite{DLSz2} which proves existence of infinitely many bounded solutions to the incompressible Euler system. This theory and propositions similar to Proposition \ref{p:ss} were used successfully to tackle problems of non-uniqueness of bounded weak solutions to the Riemann problem for isentropic compressible Euler equations, see \cite{ChiDelKre15} among others, as well as for the full Euler equations, see \cite{AKKMM}. The key lemma here is the following oscillation lemma, which provides infinitely many solutions to a linear system resembling incompressible pressureless Euler equations. 

\begin{lemma}
	Let $(\til{\vv},\til{\U})\in\R^2\times \Sz$ and $C>0$ such that  $\til{\vv}\otimes\til{\vv}-\til{\U}<\frac{C}{2} \id$. Furthermore let $\Omega\subset\R\times\R^2$ be any open set. Then there exist infinitely many maps $(\un{\vv},\un{\U})\in L^\infty(\R\times\R^2;\R^2\times\Sz)$ with the following properties.
	\begin{enumerate}
		\item $\un{\vv}$ and $\un{\U}$ vanish outside $\Omega$.
		\item The system of equations 
		\begin{align*}
			\Div  \un{\vv}\ &=\ 0,  \\
			\partial_t  \un{\vv} + \Div  \un{\U}\ &=\ 0 
		\end{align*}
		holds in the sense of distributions, i.e. for all test functions $(\varphi,\bfphi)\in \DC(\R\times\R^2;\R\times\R^2)$ it holds that 
		\begin{align*}
			\iint_\Omega\un{\vv}\cdot\Grad \varphi \dx \dt\ &=\ 0, \\
			\iint_\Omega(\un{\vv}\cdot\partial_t \bfphi + \un{\U}:\Grad \bfphi)\dx\dt\ &=\ 0. 
		\end{align*}
		\item The equation $(\til{\vv}+\un{\vv})\otimes (\til{\vv}+\un{\vv}) - (\til{\U}+\un{\U}) = \frac{C}{2} \id$ is fulfilled almost everywhere in $\Omega$. 
	\end{enumerate}
	\label{l:key}
\end{lemma}
We are not going to prove Lemma \ref{l:key} here and refer the reader to \cite[Lemma 3.7]{ChiDelKre15}, where this version of the oscillation lemma was used first.

Proposition \ref{p:ss} is an easy consequence of Lemma \ref{l:key}. In fact, Definition \ref{d:ss} of a fan subsolution was designed in such a way that adding solutions from Lemma \ref{l:key} supported in regions $\Omega_1$ and $\Omega_2$ of the fan partition to a subsolution produces solutions to the Cauchy problem \eqref{eq:euler}-\eqref{eq:Riemann}. For more details we refer to \cite[Theorem 4.2]{AKKMM}.

Since the subsolution is a piecewise constant object supposed to satisfy a system of partial differential equations, it is easy to observe that the system of PDEs reduces to a system of Rankine-Hugoniot conditions to be satisfied on the interfaces between regions where the subsolution is constant. Thus, we arrive at the following proposition.

\begin{propo} 
	Let $\vr_\pm, p_\pm\in\R^+$, $\vv_\pm\in\R^2$ be given. The constants $\mu_0,\mu_1,\mu_2\in\R$ and $\vr_i,p_i\in\R^+$,
	\begin{align*}
		\vv_i&=\left(\begin{array}{c}
		\alpha_i \\ 
		\beta_i
		\end{array}\right)\in\R^2, & \U_i&=\left(\begin{array}{rr}
		\gamma_i & \delta_i \\
		\delta_i & -\gamma_i
		\end{array}\right)\in\Sz, 
	\end{align*}
	and $C_i\in\R^+$ (for $i=1,2$) define an admissible fan subsolution to the Cauchy problem \eqref{eq:euler}-\eqref{eq:Riemann} if and only if they fulfill the following algebraic equations and inequalities: 
	\begin{itemize}
		\item Order of the speeds:
		\begin{align}
			\mu_0&<\mu_1<\mu_2
			\label{eq:order}
		\end{align}
		\item Rankine Hugoniot conditions on the left interface:
		\begin{align}
			\mu_0 (\vr_- - \vr_1) &= \vr_- v_- - \vr_1 \beta_1 \label{eq:rhl1}\\
			\mu_0 (\vr_- u_- - \vr_1 \alpha_1) &= \vr_- u_- v_- - \vr_1 \delta_1 \label{eq:rhl2}\\
			\mu_0 (\vr_- v_- - \vr_1 \beta_1) &= \vr_- v_-^2 - \vr_1 \bigg(\frac{C_1}{2}-\gamma_1\bigg) + p_- - p_1  \label{eq:rhl3} 
		\end{align}
		\begin{equation}
		\begin{split}
			&\mu_0 \bigg(\half\vr_- |\vv_-|^2 + c_v p_- - \vr_1 \frac{C_1}{2} - c_v p_1\bigg) = \\
			&\quad\bigg(\half\vr_- |\vv_-|^2 + (c_v+1) p_-\bigg) v_- - \bigg(\vr_1 \frac{C_1}{2} + (c_v+1) p_1\bigg) \beta_1
		\end{split}
		\label{eq:rhl4}
		\end{equation}
		
		\item Rankine Hugoniot conditions on the middle interface: 
		\begin{align}
			\mu_1 (\vr_1 - \vr_2) &= \vr_1 \beta_1 - \vr_2 \beta_2 \label{eq:rhm1}\\
			\mu_1 (\vr_1 \alpha_1 - \vr_2 \alpha_2) &= \vr_1 \delta_1 - \vr_2 \delta_2 \label{eq:rhm2}\\
			\mu_1 (\vr_1 \beta_1 - \vr_2 \beta_2) &= \vr_1 \bigg(\frac{C_1}{2}-\gamma_1\bigg) - \vr_2 \bigg(\frac{C_2}{2}-\gamma_2\bigg) + p_1 - p_2  \label{eq:rhm3}
		\end{align}
		\begin{equation}
		\begin{split}
			&\mu_1 \bigg(\vr_1 \frac{C_1}{2} + c_v p_1 - \vr_2 \frac{C_2}{2} - c_v p_2\bigg) = \\
			&\quad\bigg(\vr_1 \frac{C_1}{2} + (c_v+1) p_1\bigg) \beta_1 - \bigg(\vr_2 \frac{C_2}{2} + (c_v+1) p_2\bigg) \beta_2
		\end{split}
		\label{eq:rhm4}
		\end{equation}
			
		\item Rankine Hugoniot conditions on the right interface:
		\begin{align}
			\mu_2 (\vr_2 - \vr_+) &= \vr_2 \beta_2 - \vr_+ v_+ \label{eq:rhr1}\\
			\mu_2 (\vr_2 \alpha_2 - \vr_+ u_+) &= \vr_2 \delta_2 - \vr_+ u_+ v_+ \label{eq:rhr2}\\
			\mu_2 (\vr_2 \beta_2 - \vr_+ v_+) &= \vr_2 \bigg(\frac{C_2}{2}-\gamma_2\bigg) - \vr_+ v_+^2 + p_2 - p_+ \label{eq:rhr3}
		\end{align}
		\begin{equation}
		\begin{split}
			&\mu_2 \bigg(\vr_2 \frac{C_2}{2} + c_v p_2 - \half\vr_+ |\vv_+|^2 - c_v p_+\bigg) = \\
			&\quad\bigg(\vr_2 \frac{C_2}{2} + (c_v+1) p_2\bigg) \beta_2 - \bigg(\half\vr_+ |\vv_+|^2 + (c_v+1) p_+\bigg) v_+
		\end{split}
		\label{eq:rhr4}
		\end{equation}

		\item Subsolution conditions for $i=1,2$:
		\begin{align}
			C_i - \alpha_i^2 - \beta_i^2 &> 0 \label{eq:sc1}\\
			\bigg(\frac{C_i}{2} - \alpha_i^2 + \gamma_i\bigg) \bigg(\frac{C_i}{2} - \beta_i^2 - \gamma_i\bigg) - (\delta_i-\alpha_i \beta_i)^2 &> 0 \label{eq:sc2} 
		\end{align}
		\item Admissibility condition on the left interface:
		\begin{equation}
			\mu_0 \Big(\vr_1 s(\vr_1,p_1) - \vr_- s(\vr_-,p_-)\Big)\leq \vr_1 s(\vr_1,p_1) \beta_1 - \vr_- s(\vr_-,p_-) v_-
			\label{eq:adml}
		\end{equation}
		\item Admissibility condition on the middle interface: 
		\begin{equation}
			\mu_1 \Big(\vr_2 s(\vr_2,p_2) - \vr_1 s(\vr_1,p_1)\Big) \leq \vr_2 s(\vr_2,p_2) \beta_2 - \vr_1 s(\vr_1,p_1) \beta_1
			\label{eq:admm}
		\end{equation}
		\item Admissibility condition on the right interface: 
		\begin{equation}
			\mu_2 \Big(\vr_+ s(\vr_+,p_+) - \vr_2 s(\vr_2,p_2)\Big) \leq \vr_+ s(\vr_+,p_+) v_+ - \vr_2 s(\vr_2,p_2) \beta_2.
			\label{eq:admr}
		\end{equation}
		
	\end{itemize}
	\label{prop:algequa}
\end{propo} 

For the proof of Proposition \ref{prop:algequa} we refer to \cite[Proposition 4.4]{AKKMM}.

\subsection{Invariance of the Euler Equations}\label{ss:inv} 

Here we present two well-known principles of invariance for the Euler equations. The proofs of both Propositions \ref{p:Galileo} and \ref{p:Galileo2} can be carried out by a direct computation using the Euler equations \eqref{eq:euler} and the entropy condition \eqref{eq:entropy} and are left to the kind reader.

\begin{propo}\label{p:Galileo} 
Let $(\vr,\vv,p)(t,\vx)$ be an admissible weak solution to the Euler equations \eqref{eq:euler}-\eqref{eq:entropyex} with Riemann initial data
\begin{equation*}
	\begin{split}
		(\vr^0,\vv^0,p^0)(\vx)=\left\{
		\begin{array}[c]{ll}
			(\vr_-,\vv_- - \va,p_-) & \text{ if }y<0 \\
			(\vr_+,\vv_+ - \va,p_+) & \text{ if }y>0 
		\end{array}
		\right. ,
	\end{split}
\end{equation*}
where $\va\in\R^2$. Then 
\begin{equation*}
    (\vr^\new, \vv^\new, p^\new)(t,\vx) := (\vr,\vv + \va,p)(t,\vx-\va t)
\end{equation*}
is an admissible weak solution to the Euler equations \eqref{eq:euler}-\eqref{eq:entropyex} with Riemann initial data
\begin{equation*}
	\begin{split}
		(\vr^{\new,0},\vv^{\new,0},p^{\new,0})(\vx)=\left\{
		\begin{array}[c]{ll}
			(\vr_-,\vv_-,p_-) & \text{ if }y<0 \\
			(\vr_+,\vv_+,p_+) & \text{ if }y>0
		\end{array}
		\right. .
	\end{split}
\end{equation*}
\end{propo}

\begin{propo}\label{p:Galileo2}
Let $(\vr,\vv,p)(t,\vx)$ be an admissible weak solution to the Euler equations \eqref{eq:euler}-\eqref{eq:entropyex} with Riemann initial data \eqref{eq:Riemann}. Then
\begin{equation*}
    (\vr^\new, \vv^\new, p^\new)(t,\vx) := (\vr,-\vv,p)(t,-\vx)
\end{equation*}
is an admissible weak solution to the Euler equations \eqref{eq:euler}-\eqref{eq:entropyex} with Riemann initial data
\begin{equation*}
	\begin{split}
		(\vr^{\new,0},\vv^{\new,0},p^{\new,0})(\vx)=\left\{
		\begin{array}[c]{ll}
			(\vr_+,-\vv_+,p_+) & \text{ if }y<0 \\
			(\vr_-,-\vv_-,p_-) & \text{ if }y>0
		\end{array}
		\right. .
	\end{split}
\end{equation*}
\end{propo}

With these two propositions at hand we can assume without loss of generality, that $p_- < p_+$ (Proposition \ref{p:Galileo2}) and that $\vv_+ = 0$ (Proposition \ref{p:Galileo}). We will come back to this fact in Section \ref{s:proof_main}.

\section{Smallness Result}\label{s:small}

In this section we prove the following smallness result.
\begin{theorem}\label{t:small}
	Let $c_v > \frac 12$, $\vr_\pm > 0$, $0< p_- < p_+$, $u_- \in \R$ and $u_+ = v_+ = 0$. There exists $V(p_-,p_+) < (p_+-p_-)^2\frac{2c_v}{(2c_v+1)p_+ + p_-}$ such that for any $v_- > 0$ satisfying
	\begin{equation}\label{eq:smallness_cond}
	V(p_-,p_+) < \vr_-v_-^2 < (p_+-p_-)^2\frac{2c_v}{(2c_v+1)p_+ + p_-}
	\end{equation} 
	there exist infinitely many admissible weak solutions to the Riemann problem \eqref{eq:euler}-\eqref{eq:Riemann}.
\end{theorem} 

In the rest of this section we prove Theorem \ref{t:small}. According to Proposition \ref{p:ss} it suffices to show existence of an admissible fan subsolution in order to ensure that there are infinitely many admissible weak solutions to \eqref{eq:euler}-\eqref{eq:Riemann}. To this end we need to find $\mu_0,\mu_1,\mu_2\in\R$, $\vr_i,p_i\in\R^+$, $\vv_i\in\R^2$, $\U_i\in\Sz$ and $C_i\in\R^+$ ($i=1,2$) such that \eqref{eq:order}-\eqref{eq:admr} hold. These values then define an admissible fan subsolution according to Proposition \ref{prop:algequa}.

\subsection{Ansatz and Simplification}\label{ss:ansatz}

Due to the fact, that $\vv_+=0$, we make the following ansatz:
\begin{align}
\mu_2 &=  \beta_2 = v_{+} = 0  \label{eq:r.cd}\\
\alpha_1 &= \alpha_2 = u_{-}  \label{eq:r.cd2}\\
\delta_1 &= \alpha_1\beta_1 \\
\delta_2 &= 0.
\end{align}

\begin{remark}
	Note that the condition \eqref{eq:r.cd} in particular suggests that the right interface plays the role of the contact discontinuity, as the speed of the discontinuity $\mu_2$ equals the second components of the velocities of the subsolution on the left ($\beta_2$) and on the right ($v_{+}$) side of the interface. Therefore it is also natural to make the ansatz in such a way that the first component of the velocity exhibits its jump on this interface, as it is set in \eqref{eq:r.cd2}.
\end{remark}

We introduce the quantities $\ep_i$ and $\ept_i$ ($i=1,2$) as
\begin{align*}
\ep_i &:= \frac{C_i}{2} - \gamma_i - \beta_i^2 \\
\ept_i &:= C_i - \alpha_i^2 - \beta_i^2 - \ep_i 
\end{align*}
in order to simplify the subsolution inequalities (we eliminate $C_i$ and $\gamma_i$) and we denote $\beta := \beta_1$.

Then the set of algebraic equations and inequalities \eqref{eq:order}-\eqref{eq:admr} simplifies into
\begin{itemize}
	\item Order of the speeds:
	\begin{align}
		\mu_0&<\mu_1 < 0
	\label{eq:orders}
	\end{align}
	\item Rankine Hugoniot conditions on the left interface:
	\begin{align}
		\mu_0 (\vr_- - \vr_1) &= \vr_- v_{-} - \vr_1 \beta \label{eq:rhl1s}\\
		\mu_0 (\vr_- v_{-} - \vr_1 \beta) &= \vr_- v_{-}^2 - \vr_1 (\beta^2 + \ep_1) + p_- - p_1  \label{eq:rhl3s} 
	\end{align}
	\begin{equation}
	\begin{split}
		&\mu_0 \bigg(\half\vr_- v_{-}^2 + c_v p_- - \half\vr_1 (\beta^2+\ep_1+\ept_1) - c_v p_1\bigg) = \\
		&\quad\bigg(\half\vr_- v_{-}^2 + (c_v+1) p_-\bigg) v_{-} - \bigg(\half\vr_1 (\beta^2+\ep_1+\ept_1) + (c_v+1) p_1\bigg) \beta
	\end{split}
	\label{eq:rhl4s}
	\end{equation}
	
	\item Rankine Hugoniot conditions on the middle interface: 
	\begin{align}
		\mu_1 (\vr_1 - \vr_2) &= \vr_1 \beta \label{eq:rhm1s}\\
		\mu_1 \vr_1 \beta &= \vr_1 (\beta^2 + \ep_1) - \vr_2 \ep_2 + p_1 - p_2  \label{eq:rhm3s}
	\end{align}
	\begin{equation}
	\begin{split}
		&\mu_1 \bigg(\half\vr_1 (\beta^2+\ep_1+\ept_1) + c_v p_1 - \half\vr_2 (\ep_2+\ept_2) - c_v p_2\bigg) = \\
		&\quad\bigg(\half\vr_1 (\beta^2+\ep_1+\ept_1) + (c_v+1) p_1\bigg) \beta 
	\end{split}
	\label{eq:rhm4s}
	\end{equation}
	
	\item Rankine Hugoniot condition on the right interface:
	\begin{align}
		0 &= \vr_2\ep_2 + p_2 - p_+ \label{eq:rhr3s}
	\end{align}
	
	\item Subsolution conditions for $i=1,2$:
	\begin{align}
		\ep_i &> 0 \label{eq:sc1s}\\
		\ept_i &> 0 \label{eq:sc2s} 
	\end{align}
	\item Admissibility condition on the left interface:
	\begin{equation}
		\mu_0 \Big(\vr_1 s(\vr_1,p_1) - \vr_- s(\vr_-,p_-)\Big)\leq \vr_1 s(\vr_1,p_1) \beta - \vr_- s(\vr_-,p_-) v_{-}
	\label{eq:admls}
	\end{equation}
	\item Admissibility condition on the middle interface: 
	\begin{equation}
		\mu_1 \Big(\vr_2 s(\vr_2,p_2) - \vr_1 s(\vr_1,p_1)\Big) \leq - \vr_1 s(\vr_1,p_1) \beta
	\label{eq:admms}
	\end{equation}
	
\end{itemize}

The admissibility condition on the right interface is satisfied trivially as an equality.

We treat $\vr_{\pm},p_{\pm}$ and $v_{-}$ as given data and we want to find values for the 11 unknowns $\mu_0,\mu_1,\beta,p_{1,2},\vr_{1,2},\ep_{1,2}$ and $\ept_{1,2}$, such that \eqref{eq:orders}-\eqref{eq:admms} hold. 

To this end we proceed as follows:
\begin{enumerate} 
	\item First we treat $\vr_{1,2}$ and $p_{1,2}$ as parameters and determine that $\vr_- < \vr_1<\vr_K<\vr_2$ with some $\vr_K$ which is specified below.
	\item We express the other 7 remaining unknowns $\mu_0,\mu_1,\beta,\ep_{1,2}$ and $\ept_{1,2}$ as functions of $\vr_{1,2}$ and $p_{1,2}$. Hereby we simply state them. After showing that they are well-defined in Section \ref{subsec:well-def}, we prove in Section \ref{subsec:equations} that this gives in fact a solution to the equations \eqref{eq:rhl1s}-\eqref{eq:rhr3s}. 
	\item In Section \ref{subsec:order_speeds} we show that inequality \eqref{eq:orders} holds for all $\vr_{1,2}$ (where $\vr_- < \vr_1<\vr_K<\vr_2$) and $p_{1,2}$.
	\item In Section \ref{subsec:adm_cond} we choose the pressures $p_{1,2}$ as functions of $\vr_{1,2}$ and prove that with this choice the admissibility conditions \eqref{eq:admls}, \eqref{eq:admms} are fulfilled
	\item In Section \ref{subsec:subsol_cond_1} we show that the subsolution condition \eqref{eq:sc1s} is satisfied for all $\vr_{1,2}$ with $\vr_- < \vr_1 < \vr_K < \vr_2$ and $\vr_K - \vr_1$ and $\vr_2 - \vr_K$ small enough.		
	\item Finally in Section \ref{subsec:subsol_cond_2} we prove that if $\vr_-v_{-}^2 < (p_+-p_-)^2\frac{2c_v}{(2c_v+1)p_+ + p_-}$ where $(p_+-p_-)^2\frac{2c_v}{(2c_v+1)p_+ + p_-} - \vr_-v_{-}^2$ is sufficiently small, for all $\vr_{1,2}$ with $\vr_-<\vr_1<\vr_K<\vr_2$ and $\vr_K-\vr_1$ and $\vr_2-\vr_K$ sufficiently small, the subsolution condition \eqref{eq:sc2s} is fulfilled. 
\end{enumerate}

\subsection{A Particular Fan Subsolution}

Define 
\begin{equation} \label{eq:rhoK}
	\vr_K:=\vr_-\frac{p_+-p_-}{p_+-p_--\vr_-v_{-}^2}
\end{equation}
and note, that the right inequality in \eqref{eq:smallness_cond} implies
\begin{equation*}
\vr_-v_{-}^2 < (p_+-p_-)\frac{2c_vp_+ - 2c_vp_-}{2c_vp_+ + p_+ + p_-} < p_+ - p_-
\end{equation*}
and therefore $\vr_K > \vr_-$. 

For $\vr_{1,2}$ with $\vr_- < \vr_1<\vr_K<\vr_2$ we define
\begin{equation}\label{eq:betasol}
	\beta := \frac{\vr_-(\vr_2-\vr_1)}{\vr_1(\vr_2-\vr_-)}v_{-} - \frac{\sqrt{(\vr_2-\vr_1)(\vr_1-\vr_-)\big[(\vr_2-\vr_-)(p_+-p_-)-\vr_-\vr_2v_{-}^2\big]}}{\vr_1(\vr_2-\vr_-)},
\end{equation}

\begin{align}
	\mu_0 &:= \frac{\vr_1\beta - \vr_-v_{-}}{\vr_1 - \vr_-}, \label{eq:mu0sol} \\
	\mu_1 &:= -\frac{\vr_1\beta}{\vr_2-\vr_1}, \label{eq:mu1sol}
\end{align}

\begin{align}
	\ep_1 &:= \frac{1}{\vr_1} \left(p_- - p_1 + \frac{\vr_1 \vr_-}{\vr_1 - \vr_-}\big(v_- - \beta\big)^2 \right), \label{eq:ep1sol}\\ 
	\ep_2 &:= \frac{1}{\vr_2} \big(p_+ - p_2\big), \label{eq:ep2sol} 
\end{align}

\begin{align}
	\ept_1 &:= v_{-}^2 - \beta^2 -2c_v\left(\frac{p_1}{\vr_1}-\frac{p_-}{\vr_-}\right) - 2\frac{(\vr_1-\vr_-)(p_1\beta-p_-v_{-})}{\vr_-\vr_1(v_{-}-\beta)} + \frac{p_1-p_-}{\vr_1} - \frac{\vr_-(v_{-}-\beta)^2}{\vr_1-\vr_-}, \label{eq:ept1sol} \\
	\ept_2 &:= v_-^2 - 2c_v \left(\frac{p_2}{\vr_2} - \frac{p_-}{\vr_-}\right) - 2 \frac{(\vr_1 - \vr_-)(p_1\beta - p_- v_-)}{\vr_- \vr_1 (v_- - \beta)} + \frac{p_2 - p_+}{\vr_2} + 2\frac{\vr_2 - \vr_1}{\vr_1\vr_2}p_1. \label{eq:ept2sol}
\end{align}

\subsection{Well-Definition} \label{subsec:well-def}

Let us first check if $\beta,\mu_0,\mu_1,\ep_{1,2}$ and $\ept_{1,2}$ given by \eqref{eq:betasol}-\eqref{eq:ept2sol} are well-defined. Since $\vr_-<\vr_1<\vr_K<\vr_2$, it suffices to show that the term under the square root in \eqref{eq:betasol} is not negative and furthermore that $v_- -\beta\neq 0$.

Note that $\vr_2>\vr_K$ implies
$$
(\vr_2-\vr_1)(\vr_1-\vr_-)\big[(\vr_2-\vr_-)(p_+-p_-)-\vr_-\vr_2v_{-}^2\big]>0,
$$
whence $\beta$ is well-defined.

From $\vr_- < \vr_1 < \vr_2$ one simply derives that $\frac{\vr_-(\vr_2-\vr_1)}{\vr_1(\vr_2-\vr_-)}<1$. Hence we have 
\begin{equation} \label{eq:betavm}\begin{split}
\beta &= \frac{\vr_-(\vr_2-\vr_1)}{\vr_1(\vr_2-\vr_-)}v_{-} - \frac{\sqrt{(\vr_2-\vr_1)(\vr_1-\vr_-)\big[(\vr_2-\vr_-)(p_+-p_-)-\vr_-\vr_2v_{-}^2\big]}}{\vr_1(\vr_2-\vr_-)} \\
&< \frac{\vr_-(\vr_2-\vr_1)}{\vr_1(\vr_2-\vr_-)} v_- \quad<\quad v_-,
\end{split}
\end{equation}
which proves that $v_- -\beta\neq 0$ and therefore $\ept_{1,2}$ are well-defined.

\subsection{Equations} \label{subsec:equations}

We want to show that $\beta,\mu_0,\mu_1,p_{1,2},\ep_{1,2}$ and $\ept_{1,2}$ defined in \eqref{eq:betasol}-\eqref{eq:ept2sol} solve the equations \eqref{eq:rhl1s}-\eqref{eq:rhr3s}. 

We start with the observation that \eqref{eq:rhl1s} and \eqref{eq:rhm1s} easily follow from \eqref{eq:mu0sol} and \eqref{eq:mu1sol}. In addition to that one immediately verifies \eqref{eq:rhr3s} by looking at \eqref{eq:ep2sol}.

With \eqref{eq:mu0sol} and \eqref{eq:ep1sol} one shows that \eqref{eq:rhl3s} holds. It is a long but straightforward computation which checks \eqref{eq:rhm3s} from \eqref{eq:betasol}, \eqref{eq:mu1sol}, \eqref{eq:ep1sol} and \eqref{eq:ep2sol}.

The verification of \eqref{eq:rhl4s} and \eqref{eq:rhm4s} is again long but straightforward. Hereby one has to use \eqref{eq:mu0sol}, \eqref{eq:ep1sol} and \eqref{eq:ept1sol} for \eqref{eq:rhl4s} and \eqref{eq:mu1sol}-\eqref{eq:ept2sol} for \eqref{eq:rhm4s}.

\subsection{Order of the Speeds} \label{subsec:order_speeds}

Let us now check \eqref{eq:orders}. From \eqref{eq:betavm} we obtain 
$$
\vr_1 (\vr_2-\vr_-) \beta < \vr_- (\vr_2-\vr_1) v_-
$$
which is equivalent to (subtract $\vr_1^2\beta$ on both sides)
$$
(\vr_2 - \vr_1) (\vr_1\beta - \vr_- v_-) < - (\vr_1-\vr_-)\vr_1\beta.
$$
Dividing by $(\vr_2 - \vr_1)(\vr_1-\vr_-)>0$ yields $\mu_0<\mu_1$.

Furthermore it holds that $$\vr_- \vr_1 v_-^2 - (\vr_1 - \vr_-)(p_+ - p_-) > 0$$ which can be simply derived from $\vr_1 < \vr_K$. With this at hand one verifies that 
$$
\vr_-^2 (\vr_2 - \vr_1)^2 v_-^2 > (\vr_2-\vr_1)(\vr_1-\vr_-)\big[(\vr_2-\vr_-)(p_+-p_-)-\vr_-\vr_2v_{-}^2\big].
$$
Hence 
$$
\vr_- (\vr_2 - \vr_1) v_- -\sqrt{(\vr_2-\vr_1)(\vr_1-\vr_-)\big[(\vr_2-\vr_-)(p_+-p_-)-\vr_-\vr_2v_{-}^2\big]} > 0
$$
which shows that $\beta>0$. This implies that $\mu_1 < 0$.

\subsection{Admissibility Conditions} \label{subsec:adm_cond}

From now on we define the pressures $p_{1,2}$ as functions of $\vr_{1,2}$ as follows

\begin{align} \label{eq:p1sol}
p_1 &:= p_- \left(\frac{\vr_1}{\vr_-}\right)^{\frac{c_v+1}{c_v}}, \\ \label{eq:p2sol}
p_2 &:= p_- \left(\frac{\vr_2}{\vr_-}\right)^{\frac{c_v+1}{c_v}}.
\end{align}

This implies 
$$
s(\vr_1,p_1)=s(\vr_2,p_2)=s(\vr_-,p_-)
$$
and hence the inequalities \eqref{eq:admls}, \eqref{eq:admms} hold as equalities.

Now we are left with two parameters $\vr_{1,2}$ and we have expressed the other unknowns in terms of these parameters.

\subsection{Subsolution Condition, Part 1}\label{subsec:subsol_cond_1}

Here we show that the subsolution condition \eqref{eq:sc1s} holds for $\vr_{1,2}$ with $\vr_-<\vr_1<\vr_K<\vr_2$ and $\vr_K-\vr_1$ and $\vr_2-\vr_K$ sufficiently small.

In order to prove that $\ep_2>0$ we have to show that $p_+ - p_2 > 0$ according to \eqref{eq:ep2sol}. We compute
\begin{align*}
	\lim\limits_{\vr_2\to\vr_K}\big(p_+  - p_2\big) &= \lim\limits_{\vr_2\to\vr_K} \left(p_+ - p_- \left(\frac{\vr_2}{\vr_-}\right)^{\frac{c_v + 1}{c_v}} \right) \\
	&=p_+ - p_- \left(\frac{\vr_K}{\vr_-}\right)^{\frac{c_v + 1}{c_v}} \\
	&=p_+ - p_- \left(\frac{p_+-p_-}{p_+-p_--\vr_-v_{-}^2} \right)^{\frac{c_v + 1}{c_v}}
\end{align*}

The right inequality in \eqref{eq:smallness_cond} implies that 
\begin{align*}
	p_+ - p_- - \vr_- v_-^2 &> p_+ - p_- - \frac{2c_v (p_+-p_-)^2}{(2c_v + 1)p_+ + p_-} \\
	&= \frac{p_+ - p_-}{(2c_v + 1)p_+ + p_-} \Big((2c_v + 1)p_+ + p_- - 2c_v (p_+ - p_-)\Big) \\
	&= (p_+ - p_-) \frac{(2c_v + 1)p_- + p_+}{(2c_v + 1)p_+ + p_-}.
\end{align*}

Hence 
\begin{align*}
p_+ - p_- \left(\frac{p_+-p_-}{p_+-p_--\vr_-v_{-}^2} \right)^{\frac{c_v + 1}{c_v}} &> p_+ - p_- \left( \frac{(2c_v + 1)p_+ + p_-}{(2c_v + 1)p_- + p_+}\right)^{\frac{c_v +  1}{c_v}} \\ 
&= p_- \left( \frac{p_+}{p_-} - \left( \frac{(2c_v + 1)\frac{p_+}{p_-} + 1}{(2c_v + 1) + \frac{p_+}{p_-}}\right)^{\frac{c_v +  1}{c_v}} \right).
\end{align*}

\begin{lemma} \label{l:f}
	Define 
	$$
	f(t) := t - \left( \frac{(2c_v + 1)t + 1}{(2c_v + 1) + t}\right)^{\frac{c_v +  1}{c_v}}.
	$$
	Then $f(t)>0$ for all $t>1$.
\end{lemma}

\begin{proof}
	It is simple to check that $f(1) = 0$. The derivatives of $f$ read
	\begin{equation*}
	\begin{split}
		f'(t)&= 1-4(c_v+1)^2 \frac{\left((2c_v+1)t+1\right)^{\frac{1}{c_v}}}{((2c_v + 1) +t)^{2+\frac{1}{c_v}}}\\
		f''(t)&=  8(c_v+1)^2 (2c_v+1) \frac{((2c_v+1)t+1)^{\frac{1}{c_v} -1}}{((2c_v+1)+t)^{3+\frac{1}{c_v}}}(t-1)
	\end{split}
	\end{equation*}
	and hence we have $f'(1)=f''(1) = 0$ and $f''(t)>0$ for $t>1$. This gives the claim.
\end{proof}

Since we have $0 < p_- < p_+$, Lemma \ref{l:f} implies that 
$$
\lim\limits_{\vr_2\to\vr_K}\big(p_+  - p_2\big) > p_- f\left(\frac{p_+}{p_-}\right) > 0,
$$
and hence by continuity $\ep_2 >0$ for all $\vr_2>\vr_K$ with $\vr_2-\vr_K$ sufficiently small. 

It is easy to verify that $\lim\limits_{\vr_1\to\vr_K} \beta = 0$ for all $\vr_2>\vr_K$. With this at hand we obtain
$$
\lim\limits_{\vr_1\to\vr_K} \left(p_- + \frac{\vr_1\vr_-}{\vr_1-\vr_-} (v_- - \beta)^2\right) = p_- + \frac{\vr_K\vr_-}{\vr_K-\vr_-} v_-^2 = p_+. 
$$
Hence the computations above yield 
$$
\lim\limits_{\vr_1\to\vr_K} \left(p_- - p_1 + \frac{\vr_1 \vr_-}{\vr_1 - \vr_-}\big(v_- - \beta\big)^2 \right)>0
$$
and again by continuity we have $\ep_1>0$ for all $\vr_1<\vr_K$ with $\vr_K-\vr_1$ sufficiently small.

\subsection{Subsolution Condition, Part 2} \label{subsec:subsol_cond_2}

To finish the proof of Theorem \ref{t:small} it remains to show that if $\vr_-v_{-}^2 < (p_+-p_-)^2\frac{2c_v}{(2c_v+1)p_+ + p_-}$ where $(p_+-p_-)^2\frac{2c_v}{(2c_v+1)p_+ + p_-} - \vr_-v_{-}^2$ is sufficiently small, for all $\vr_{1,2}$ with $\vr_-<\vr_1<\vr_K<\vr_2$ and $\vr_K-\vr_1$ and $\vr_2-\vr_K$ sufficiently small, the subsolution condition \eqref{eq:sc2s} is fulfilled. 

To this end we consider the formulas for $\ept_{1,2}$, i.e. \eqref{eq:ept1sol}, \eqref{eq:ept2sol} and replace $\vr_{1,2}$ by $\vr_K$ and $ \vr_-v_{-}^2$ by $$(p_+-p_-)^2\frac{2c_v}{(2c_v+1)p_+ + p_-}.$$
For both $\ept_1$ and $\ept_2$ we will end up with the same quantity denoted by $\Et$. We will show that $\Et>0$ which implies the claim by continuity. 

We begin with $\ept_1$ \eqref{eq:ept1sol} and replace $\vr_{1,2}$ by $\vr_K$. The result is denoted by $\Et_1$ and we obtain 
\begin{equation} \label{eq:Et1} 
	\Et_1 = v_-^2 - 2c_v\left(\frac{p_K}{\vr_K} - \frac{p_-}{\vr_-}\right) + 2\frac{(\vr_K - \vr_-) p_-}{\vr_- \vr_K} + \frac{p_K - p_-}{\vr_K} - \frac{\vr_- v_-^2}{\vr_K - \vr_-}
\end{equation}
where $$p_K:= p_- \left(\frac{\vr_K}{\vr_-}\right)^{\frac{c_v + 1}{c_v}}.$$

Replacing $\vr_{1,2}$ by $\vr_K$ in the formula for $\ept_2$ \eqref{eq:ept2sol} as well, we obtain 
\begin{equation} \label{eq:Et2} 
	\Et_2 = v_-^2 - 2c_v \left(\frac{p_K}{\vr_K} - \frac{p_-}{\vr_-}\right) + 2 \frac{(\vr_K - \vr_-)p_-}{\vr_- \vr_K} + \frac{p_K - p_+}{\vr_K}.
\end{equation}
An easy computation shows $$\frac{p_K - p_-}{\vr_K} - \frac{\vr_- v_-^2}{\vr_K - \vr_-} = \frac{p_K - p_+}{\vr_K}$$
and hence $\Et_1 = \Et_2$.

Replacing $\vr_K$ in \eqref{eq:Et1} with the help of \eqref{eq:rhoK} we end up with
\begin{equation} \label{eq:Et1ex}
	\Et_1 = \Et_2 = \frac{p_-}{\vr_-}\left[ (1-2c_v) \left(\frac{p_+ - p_-}{p_+ - p_- - \vr_-v_-^2}\right)^{\frac{1}{c_v}} + 2c_v + \frac{\vr_- v_-^2 (p_- + 2p_+) - p_+(p_+ - p_-)}{p_-(p_+ - p_-)}\right].
\end{equation}
Let us now replace $\vr_-v_{-}^2$ by $(p_+-p_-)^2\frac{2c_v}{(2c_v+1)p_+ + p_-}$ in \eqref{eq:Et1ex}. The resulting quantity is denoted by $\Et$. We obtain
\begin{align*}
\Et &= \frac{p_-(2c_v-1)\left((2c_v+1)p_- + p_+\right)}{\vr_-\left((2c_v+1)p_+ + p_-\right)} \left[\frac{p_+}{p_-} - \left(\frac{(2c_v+1)\frac{p_+}{p_-}+1}{(2c_v+1)+\frac{p_+}{p_-}}\right)^{\frac{c_v + 1}{c_v}}\right] \\
&=\frac{p_-(2c_v-1)\left((2c_v+1)p_- + p_+\right)}{\vr_-\left((2c_v+1)p_+ + p_-\right)} \  f\left(\frac{p_+}{p_-}\right)
\end{align*}
with the same $f$ as in Lemma \ref{l:f}. Since by assumption $c_v>\half$ and because of Lemma \ref{l:f} (note that $p_+>p_-$ by assumption) we deduce $\Et > 0$.

\section{Proof of Theorem \ref{t:main}}\label{s:proof_main} 

Before we continue further, let us reiterate that without loss of generality we can assume $p_- < p_+$. Indeed, if that is not the case, we use the invariance of the Euler system stated in Proposition \ref{p:Galileo2}. We split Theorem \ref{t:main} into four cases.
\begin{itemize}
    \item[(1)] The self-similar solution contains two shocks.
    \item[(2)] The self-similar solution contains one shock and one rarefaction wave and the condition
    \begin{equation}\label{eq:smallcondthm}
        \sqrt{\frac{V(p_-,p_+)}{\vr_-}} < v_- - v_+ < (p_+-p_-)\sqrt{\frac{2c_v}{\vr_-((2c_v+1)p_++p_-)}}
    \end{equation}
    holds, where $V(p_-,p_+)$ is given by Theorem \ref{t:small}.
    \item[(3)] The self-similar solution contains one shock and one rarefaction wave and the condition \eqref{eq:smallcondthm} does not hold. To be precise, this means that the left inequality in \eqref{eq:smallcondthm} is not fulfilled.
    \item[(4)] The self-similar solution contains exactly one shock and no rarefaction wave.
\end{itemize}

The case (1) was solved in \cite[Theorem 3.5]{AKKMM}.

The cases (2), (3) and (4) will be handled using Theorem \ref{t:small}, the Galilean invariance stated in Proposition \ref{p:Galileo} and the patching procedure introduced in \cite{MarKli17}, where the latter is not required in case (2). Details are presented in the following sections.

\subsection{Case (2)} \label{ss:case2}

In order to prove Theorem \ref{t:main} in case (2), it is enough to use Theorem \ref{t:small} and the Galilean invariance stated in Proposition \ref{p:Galileo}. Indeed Theorem \ref{t:small} yields infinitely many solutions to the problem with left state $(\vr_-,\vv_- - \vv_+,p_-)$ and right state $(\vr_+,0,p_+)$. Finally Proposition \ref{p:Galileo} states that these solutions can be shifted by $\vv_+$ to obtain solutions to the problem with left state $(\vr_-,\vv_-,p_-)$ and right state $(\vr_+,\vv_+,p_+)$.

\subsection{Case (3)} \label{ss:case3} 

In order to prove Theorem \ref{t:main} in case (3), we proceed as follows. We start with identifying the middle states of the 1D self-similar solution, $(\vr_{M-},\vv_{M-},p_M)$ and $(\vr_{M+},\vv_{M+},p_M)$, where $u_{M-} = u_-$, $u_{M+} = u_+$ and $v_{M-} = v_{M+} =: v_{M}$. We recall that this means that states $(\vr_-,\vv_-,p_-)$ on the left and $(\vr_{M-},\vv_{M-},p_M)$ on the right are connected by an admissible 1-shock and states $(\vr_{M+},\vv_{M+},p_M)$ on the left and $(\vr_+,\vv_+,p_+)$ on the right are connected by a 3-rarefaction wave. Note that the 1D theory also determines that $v_{M} < \min\{v_-,v_+\}$ and $p_- < p_M < p_+$.

Let $\delta > 0$. We denote $p^\delta := p_M + \delta$ and require $\delta$ small such that $p^\delta < p_+$. With $p^\delta$ given we compute $\vr^\delta$ and $v^\delta$ as
\begin{align*}
	v^\delta &:= v_+ - 2\sqrt{c_v(c_v+1)}\sqrt{\frac{p_+}{\vr_+}}\left(1-\left(\frac{p^\delta}{p_+}\right)^\frac{1}{2(c_v+1)}\right),\\
	\vr^\delta &:= \vr_+ \left(\frac{p^\delta}{p_+}\right)^{\frac{c_v}{c_v + 1}}, 
\end{align*}
which implies that the states $(\vr^\delta,\vv^\delta = (u_+,v^\delta),p^\delta)$ on the left and $(\vr_+,\vv_+,p_+)$ on the right can be connected by a 3-rarefaction wave. 

It is easy to observe that 
\begin{equation} \label{eq:vdelta}
v_- - v^\delta < (p^\delta - p_-)\sqrt{\frac{2c_v}{\vr_-(p_-+(2c_v+1)p^\delta)}}.  
\end{equation} 

From Theorem \ref{t:small} we obtain for each $p^\delta\in[p_M,p_+]$ a $V(p_-,p^\delta)< (p^\delta - p_-)^2\frac{2c_v}{(2c_v+1)p^\delta + p_-}$. For  $p^\delta=p_M$ we get
\begin{equation} \label{eq:VpM}
\begin{split}
	V(p_-,p_M) &<  (p_M - p_-)^2\frac{2c_v}{(2c_v+1)p_M + p_-} \\
	&= \vr_- (v_- - v_M)^2,
\end{split}
\end{equation}
where the latter equality is a consequence of the fact that the states $(\vr_-,\vv_-,p_-)$ and $(\vr_{M-},\vv_{M-},p_M)$ are connected by a 1-shock. Note that we may assume without loss of generality that the map $\delta \mapsto V(p_-,p^\delta)$ is continuous. Note furthermore that for $p^\delta=p_M$ we obtain $v^\delta=v_M$ and that the map $\delta \mapsto v^\delta$ is continuous as well. Hence inequality \eqref{eq:VpM} implies that 
\begin{equation} \label{eq:Vpdelta}
V(p_-,p^\delta) < \vr_- (v_- - v^\delta)^2
\end{equation}
as long as $\delta$ is sufficiently small. Since for sufficiently small $\delta$ we have $v_- - v^\delta >0$, we can combine \eqref{eq:vdelta} and \eqref{eq:Vpdelta} to
\begin{equation*}
\sqrt{\frac{V(p_-,p^\delta)}{\vr_-}} <  v_- - v^\delta < (p^\delta - p_-)\sqrt{\frac{2c_v}{\vr_-(p_-+(2c_v+1)p^\delta)}}.  
\end{equation*}

A sketch of the involved states in the phase space and the shock and rarefaction curves can be found in Figure \ref{fig:phasespace}.

\begin{figure}[hbt] 
	\centering
	\includegraphics[width=0.9\textwidth]{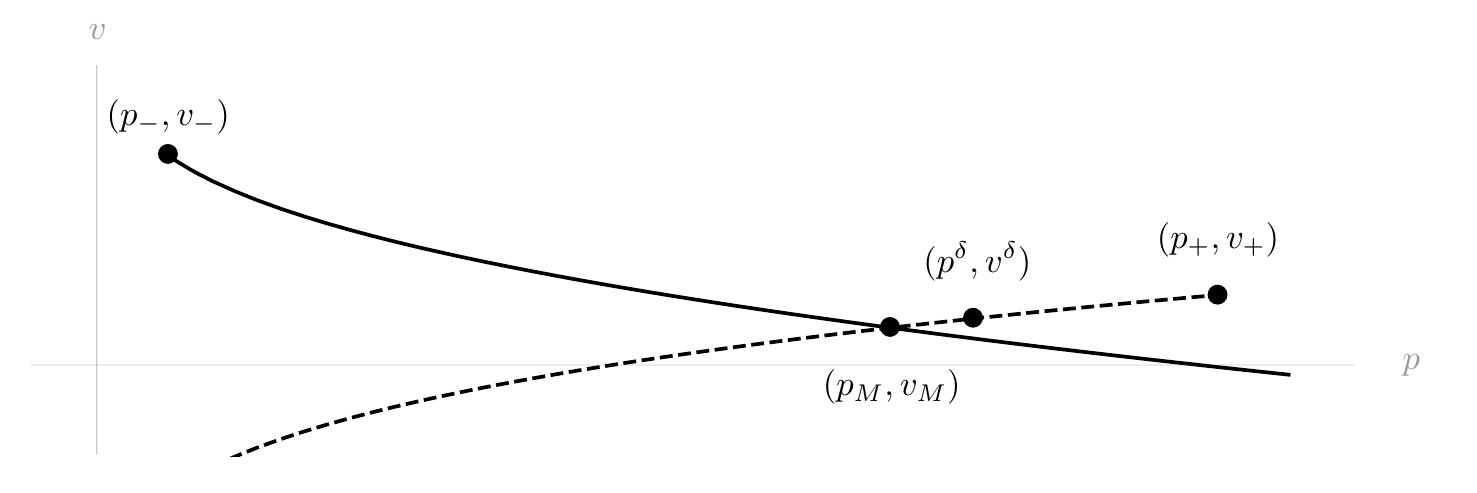}
	\caption{Projection of the 1-shock curve (solid line) of state $(\vr_-,v_-,p_-)$ and the 3-rarefaction curve (dashed line) of state $(\vr_+,v_+,p_+)$ to the $p-v$-plane. The intersection of these two curves represents pressure $p_M$ and velocity $v_M$ of the middle states. Theorem \ref{t:small} can be applied if the state $(\vr_+,v_+,p_+)$ is close to the 1-shock curve. Since this is not the case, we introduce the auxiliary state $( \vr^\delta,v^\delta,p^\delta)$ which can be connected to $(\vr_+,v_+,p_+)$ by a 3-rarefaction wave.}
	\label{fig:phasespace}
\end{figure}

Now we shift all appearing velocities by $\vv^\delta := (u_+,v^\delta)$ and use Theorem \ref{t:small} to obtain infinitely many solutions to the problem with left state $(\vr_-,\vv_--\vv^\delta,p_-)$ and right state $(\vr^\delta,0,p^\delta)$. 

In order to patch together such a solution and the 3-rarefaction wave connecting the states $(\vr^\delta,0,p^\delta)$ on the left and $(\vr_+,\vv_+-\vv^\delta,p_+)$ on the right, we have to check if the rarefaction wave does not interfere with the regions $\Omega_1,\Omega_2$ of the fan partition related to the solutions given by Theorem \ref{t:small}. In other words we have to show that the rarefaction wave lies in the region $\Omega_+$, where the solutions given by Theorem \ref{t:small} are constant. It is well-known (see classical monographs \cite{Dafermos16} or \cite{Smoller67}) that the left borderline of the 3-rarefaction wave is equal to $\lambda_3(\vr^\delta,0,p^\delta)$, where $\lambda_3$ is the third characteristic speed of the Euler system expressed in \eqref{eq:eigenvalues}. We immediately observe
\begin{equation*}
\lambda_3(\vr^\delta,0,p^\delta) = \sqrt{\frac{c_v+1}{c_v}\frac{p^\delta}{\vr^\delta}} > 0.
\end{equation*}
Since for the right borderline of the region $\Omega_2$ (which was denoted by $\mu_2$) we have $\mu_2=0$ according to \eqref{eq:r.cd}, we conclude the desired relation $\mu_2 < \lambda_3(\vr^\delta,0,p^\delta)$. Hence we have infinitely many solutions to the problem with $(\vr_-,\vv_- - \vv^\delta,p_-)$ as left state and $(\vr_+,\vv_+ - \vv^\delta,p_+)$ as right state, an example of which is shown in Figure \ref{fig:solution}. 

\begin{figure}[hbt] 
	\centering
	\includegraphics[width=0.7\textwidth]{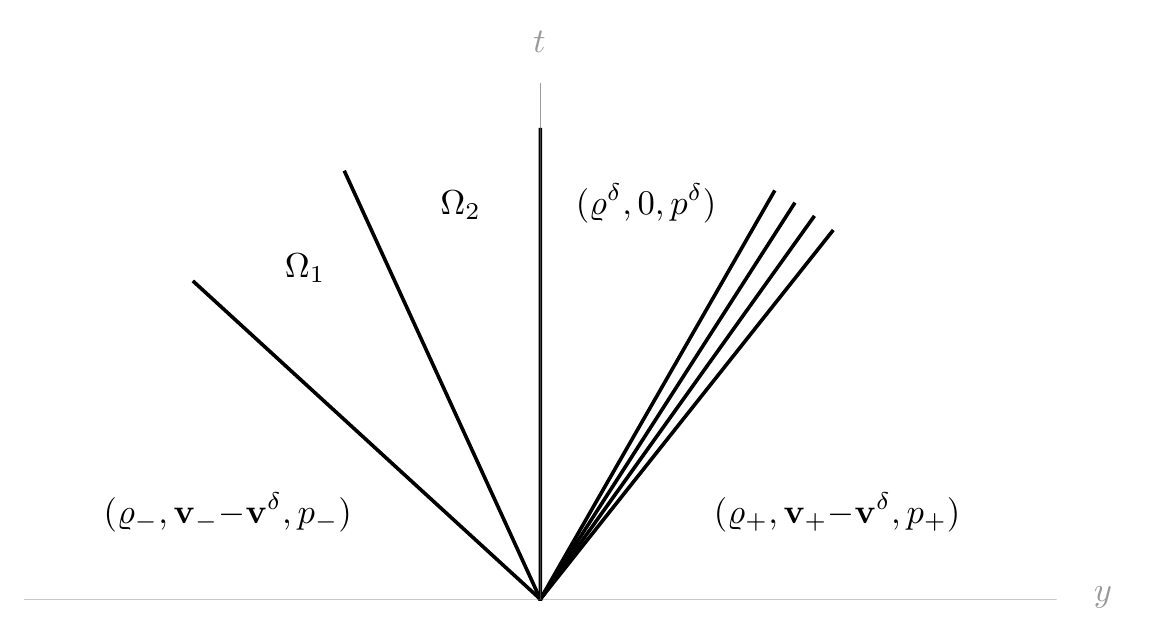} 
	\caption{Example of a solution to the problem with left state $(\vr_-,\vv_- - \vv^\delta,p_-)$ and right state $(\vr_+,\vv_+ - \vv^\delta,p_+)$ as constructed in Section \ref{ss:case3}. The solutions constructed in Section \ref{ss:case4} are similar. The only difference is that the rarefaction is replaced by a shock.} 
	\label{fig:solution}
\end{figure}

Finally we use the Galilean invariance stated in Proposition \ref{p:Galileo} to obtain infinitely many admissible weak solutions to the original problem with left state $(\vr_-,\vv_-,p_-)$ and right state $(\vr_+,\vv_+,p_+)$.

\subsection{Case (4)} \label{ss:case4} 

The proof of Theorem \ref{t:main} in case (4) is similar to the one in case (3) presented in Section \ref{ss:case3}. Since the self-similar solution contains only an admissible 1-shock and a possible contact discontinuity, the states $(\vr_-,\vv_-,p_-)$ and $(\vr_+,\vv_+,p_+)$ satisfy
\begin{equation*}
v_- - v_+ = (p_+ - p_-)\sqrt{\frac{2c_v}{\vr_-(p_-+(2c_v+1)p_+)}}.  
\end{equation*}

Again we take $\delta > 0$ sufficiently small and denote now $p^\delta := p_+ + \delta$. We compute  $\vr^\delta$ and $v^\delta$ as 
\begin{align}
	v^\delta &:= v_+ +  (p^\delta - p_+)\sqrt{\frac{2c_v}{\vr_+(p_+ + (2c_v+1)p^\delta)}}, \label{eq:4vdelta}\\
	\vr^\delta &:=\vr_+ \frac{(2c_v +1) p^\delta + p_+}{(2c_v +1) p_+ + p^\delta}. \label{eq:4rhodelta}
\end{align}
Then $(\vr^\delta,\vv^\delta=(u_+,v^\delta),p^\delta)$ on the left and $(\vr_+,\vv_+,p_+)$ on the right can be connected by an admissible 3-shock. 

With similar arguments as in Section \ref{ss:case3} we obtain
\begin{equation*}
\sqrt{\frac{V(p_-,p^\delta)}{\vr_-}} <  v_- - v^\delta< (p^\delta - p_-)\sqrt{\frac{2c_v}{\vr_-(p_-+(2c_v+1)p^\delta)}}  
\end{equation*}
as long as $\delta > 0$ is small enough, where $V$ is given by Theorem \ref{t:small}. 

Again we show a sketch of the setting in the phase space, see Figure \ref{fig:phasespace-shock}.

\begin{figure}[hbt] 
	\centering
	\includegraphics[width=0.9\textwidth]{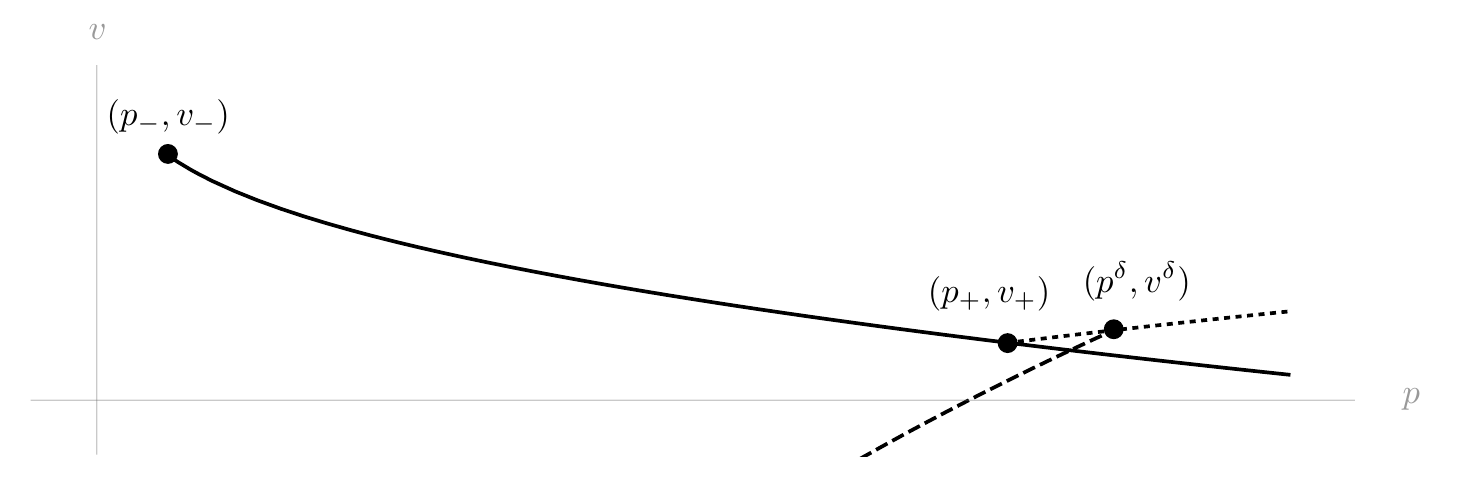}
	\caption{Projection of the 1-shock curve (solid line) of state $(\vr_-,v_-,p_-)$, the 3-shock curve (dotted line) of state $(\vr_+,v_+,p_+)$ and the 3-rarefaction curve (dashed line) of state $(\vr^\delta,v^\delta,p^\delta)$ to the $p-v$-plane.}
	\label{fig:phasespace-shock}
\end{figure}

We shift all appearing velocities by $\vv^\delta = (u_+,v^\delta)$. Then we use Theorem \ref{t:small} to obtain infinitely many admissible weak solutions to the problem with left state $(\vr_-,\vv_--\vv^\delta,p_-)$ and right state $(\vr^\delta,0,p^\delta)$. 

Our aim is again to patch together such a solution and the admissible 3-shock connecting the states $(\vr^\delta,0,p^\delta)$ on the left and $(\vr_+,\vv_+ - \vv^\delta, p_+)$ on the right. To this end we have to check if the 3-shock does not interfere with the regions $\Omega_1,\Omega_2$ of the fan partition related to the solutions given by Theorem \ref{t:small}. With the Rankine-Hugoniot condition the shock speed $\sigma$ is given by 
$$
\sigma = \frac{- \vr_+(v_+ - v^\delta)}{\vr^\delta - \vr_+}
$$
From \eqref{eq:4rhodelta} we obtain $\vr^\delta>\vr_+$, and from \eqref{eq:4vdelta} we get $v^\delta>v_+$. With this at hand we deduce the desired relation $\sigma>0$.

This yields infinitely many admissible weak solutions to the problem with left state $(\vr_-,\vv_- - \vv^\delta,p_-)$ and right state $(\vr_+,\vv_+ - \vv^\delta,p_+)$, which finishes -- together with the Galilean invariance (Proposition \ref{p:Galileo}) -- the proof.

\section{Concluding Remarks}\label{s:conclude}

\begin{remark}\label{r:multiD}
The results of this paper naturally extend to any space dimension larger than 1.
\end{remark}

\begin{remark}\label{r:CD2} 
As we pointed out in the Introduction and in Section \ref{ss:ansatz}, our ansatz for the subsolution is made in such a way, that the right interface plays the role of the contact discontinuity. This is somewhat counterintuitive and it seems more natural to look for subsolutions where the middle interface plays the role of the contact discontinuity, as it was for subsolutions constructed in \cite{AKKMM} for the case of self-similar solutions containing two shocks. However, despite our efforts we were not able to prove that any such subsolutions exist if the self-similar solution contains just one shock and one rarefaction wave.
\end{remark}

\begin{remark}\label{r:CDopen} 
Similarly as it is in the case of the isentropic Euler system with power law pressure, it remains an open question whether a self-similar solution consisting only of a single contact discontinuity or a contact discontinuity together with rarefaction waves is unique in the set of multi-dimensional bounded admissible weak solutions.
\end{remark}

\begin{remark}\label{r:summary} 
	Let us summarize the (non-)uniqueness results to our problem. The possible structures of the 1D self-similar solution are shown in Table \ref{table}. The 1D self-similar solution is unique in cases 1, 3, 7 and 9. There exist infinitely many admissible weak solutions in cases 2, 4, 5, 6, 8, 11, 13, 14, 15 and 17. As already pointed out in Remark \ref{r:CDopen}, the remaining cases 10, 12, 16 and 18 are open.
\end{remark}

\section*{Acknowledgements}
The research leading to these results was partially supported by the European Research Council under the European Union's Seventh Framework Programme (FP7/2007-2013)/ ERC Grant Agreement 320078. O. Kreml and V. M\'acha were supported by the GA\v CR (Czech Science Foundation) project GJ17-01694Y in the general framework of RVO: 67985840.

\end{document}